\documentclass[10pt, letterpaper]{article}
\usepackage[utf8]{inputenc}

\usepackage[english]{babel}
\usepackage{graphicx}
\usepackage{mathtools}
\usepackage{amsmath}
\usepackage{amssymb}
\usepackage{breakcites}

\usepackage{comment}

\usepackage{amsthm}
\usepackage{hyperref}
\usepackage{subfigure}
\usepackage{tikz}
 \usepackage{algorithm}

\usepackage{algorithmic}

\usetikzlibrary{positioning,chains,fit,shapes,calc}
\usetikzlibrary{backgrounds}
\usetikzlibrary{arrows.meta}
\usetikzlibrary{shapes,snakes}

\definecolor{carmine}{rgb}{0.59, 0.0, 0.09}

\newtheorem{theorem}{Theorem}
\newtheorem{lemma}[theorem]{Lemma}
\newtheorem{proposition}[theorem]{Proposition}

\definecolor{ao}{rgb}{0.55, 0.71, 0.0}
\definecolor{bleudefrance}{rgb}{0.19, 0.55, 0.91}
\definecolor{dimgray}{rgb}{0.41, 0.41, 0.41}    
\definecolor{mediumorchid}{rgb}{0.73, 0.33, 0.83}
\definecolor{mediumtealblue}{rgb}{0.0, 0.33, 0.71}
\definecolor{harvestgold}{rgb}{0.85, 0.57, 0.0}
\definecolor{blue(pigment)}{rgb}{0.2, 0.2, 0.6}
\definecolor{forestgreen(traditional)}{rgb}{0.27, 0.35, 0.27}
\definecolor{cadmiumred}{rgb}{0.89, 0.0, 0.13}
\definecolor{orange(webcolor)}{rgb}{1.0, 0.5, 0.0}
\definecolor{tangerine}{rgb}{0.95, 0.52, 0.0}

\title{Minimum jointly structural input and output selection for strongly connected networks\footnote{
{\scriptsize G. Ramos and A. Pedro Aguiar are with the Department of Electrical and Computer Engineering, Faculty of Engineering, University of Porto, Portugal. Pequito is a faculty member at the Delft University of Technology in the Delft Center for Systems and Control.
 This work was supported in part by projects: IMPROVE (POCI-01-0145-FEDER-031823), funded by FEDER funds through COMPETE2020 – POCI and by the Portuguese national FCT/MCTES (PIDDAC); RELIABLE (PTDC/EEI-AUT/3522/2020) funded by FCT/MCTES; and DynamiCITY (NORTE-01-0145-FEDER-000073), funded by NORTE2020/PORTUGAL2020, through the European Regional Development Fund.}}
}
\date{May 20, 2021}

\author{
Guilherme Ramos, A. Pedro Aguiar, Sérgio Pequito}

\begin{document}
\maketitle

\begin{abstract}
In this paper, given a linear time-invariant strongly connected network, we study the problem of determining the minimum number of state variables that need to be simultaneously actuated and measured to ensure structural controllability and observability, respectively.  
This problem is fundamental in the design of multi-agent systems, where there are economic constraints in the decision of which agents to equip with a more costly on-board system  that will allow the agent to have both actuation and sensing capabilities. 
Despite the combinatorial nature of this problem, we present a solution that couples the design of both structural controllability and structural observability counterparts to address it with polynomial-time complexity. 
\end{abstract}

\section{Introduction}\label{sec:intro}

Multi-agent dynamical systems (MADS) 
can resolve problems that are challenging or unsuitable for solving either with a single agent or a monolithic system~\cite{dorri2018multi}. 
These systems emerge in a plethora of applications, including consensus problems~\cite{9304107,ramosIJoC2020}, target surveillance~\cite{hu2019distributed}, online trading~\cite{luo2018distributed}, network resistance~\cite{RAMOS2021104842}, disaster response~\cite{nadi2017adaptive}, and wireless sensor networks (WSN)~\cite{akyildiz2002wireless}, just to name a few.  

Two systems properties that are desirable in MADS are controllablility and observability, that enable the proper regulation and monitoring of the agents behavior. 
When dealing with large-scale MADS, we may need to equip a subset of agents with more expensive on-board capabilities to equip them with actuation and sensing capabilities. For instance, these often rely on long-range communication system to exchange both actuation and sensing information. 

A recurrent scheme for surveillance, exploration, and measuring tasks considers a multi-agent system composed of vehicles interconnected by a communication network. 
Missions involving the use of an extensive amount of such vehicles may adopt a leader/followers quest \cite{rribeiro:controlo,rribeiro:med}. 
For instance, in the following scenarios: 
\textit{(i)} expensive nodes (leaders) that can communicate with a ground station to receive mission commands and that need to be equipped with complex sensors or localization devices; and  
\textit{(ii)} cheaper drones (followers) executing local controllers based on onboard sensors that measure relative localization and receive a small amount of data from the leaders. 
Therefore, for budgetary reasons, a crucial task is to minimize the number of leaders, without compromising the overall system controllability and observability.

Furthermore, envisioned operational scenarios in search and rescue applications and environmental monitoring using autonomous robotic vehicles require mobile multi-agent systems with complementary sensor suites to increase task efficiency and performance. One example for the former case arises when there is a cooperation between heterogeneous unmanned aerial vehicles (UAVs)~\cite{kaminer2007coordinated}, where only one UAV carries on-board expensive sensors like infrared cameras or a LIDAR sensor. For the latter case, an example occurs in some marine applications that may include ambient data acquisition, pollution source localization, and mapping. In this case, some marine robotic vehicles may carry more sophisticated and high-performance sensor suites (that usually require some latency time to detect particles in water) than others.

In recent years, research has focused mainly in determining a solution for the minimal controllability problem (or, by duality between controllability and observability, the minimal observability problem)~\cite{olshevsky2014minimal,pequito2017robust,ramos2018robust}. 
Recently, the authors in~\cite{ramos2021robust} show that the minimum jointly input and output selection problem is NP-complete, and proposed  efficient polynomial-time algorithms to compute approximate solutions.   

Notwithstanding, in the context of MADS, we have the freedom of selecting the dynamics weights that would account for the communication protocol between the agents.
We propose to leverage structural systems  theory, that enables a parametric (i.e., a structure-based) approach to the minimum jointly input and output selection problem~\cite{ramos2020structural}. Structural counterparts of controllability and observability hold for almost all parametric choices in infinite fields. Furthermore, they leverage graph-theoretical characterizations in the context of efficient minimum actuator/sensor placement
~\cite{lin1974structural,pequito2015framework,ramos2020structural}. 

In this work, we propose a novel problem formulation and solution with potential implications in designing engineering systems. 
Furthermore, whereas insights from directly related problems (e.g., sparsest input/output structural controllability/observability~\cite{pequito2015framework}) are useful, the direct use of these approaches do not allow to solve the proposed  problem (i.e., they will lead to suboptimal solutions, as illustrated in the examples of Section~\ref{sec:ill_exp}). 
That said, under mild assumption on the network structure (strongly connected networks), a key contribution of this paper is the derivation of adequate transformations needed to reduce the problem to a combinatorial problem that can be efficiently solved using a maximum weight maximum matching, in which construction and weights are tailored to solve the proposed problem, with the formal proof presented in Theorem~\ref{th:soundness}. 
Furthermore, it is worth mention that the proposed reduction would not allow us to solve the sparsest input/output structural controllability/observability problems.   

In summary, we seek  to address the following research question.

\begin{description}
\item[$\overline{\textsf{\textbf{RQ}}}$] How can we efficiently find a minimal sensor and actuator placement sharing the maximum possible state variables that ensures system's structural controllability and observability? 
\end{description}

We organized the remainder of the paper as follows. 
In Section~\ref{sec:prob_stat}, we formally state the problem that we address in Section~\ref{sec:main}. 
Subsequently, we illustrate the proposed algorithm with examples in Section~\ref{sec:ill_exp}. 
Section~\ref{sec:conclusion} concludes the paper and sheds light on future research directions.


\subsection*{Notation}\label{sec:notation}

We denote the set of real numbers by $\mathbb R$ and the set of integers by $\mathbb Z$. Moreover, we denote by $\mathbb Z_0^+$ the set of non-negative integers. 

We denote matrices by upper-case letters, e.g., $A,B$ and $C$. Similarly, we denote vectors by lower-case letters, e.g., $x,y$ and $u$. For a vector $x\in\mathbb R^n$, we denote its $i$-th entry as $x_i$, where $i\in\{1,\ldots,n\}$ and, analogously, for a matrix $A\in\mathbb R^{n\times m}$, we denote the $i$-th row of $A$ by $A_i$ and the $j$-th entry of the $i$-th row by $A_{ij}$, where $i\in\{1,\ldots,n\}$ and $j\in\{1,\ldots,m\}$.  
We denote the identity matrix of size $n$ by $\mathbb I_n$. 
Given $A_1\in\mathbb R^{n\times m_1}$ and $A_2\in\mathbb R^{n\times m_2}$, we define by $[A_1,A_2]\in\mathbb R^{n\times (m_1+m_2)}$ the matrix whose first $m_1$ columns are the columns of $A_1$ and the last $m_2$ columns are the columns of $A_2$. 
Similarly, given $A_1\in\mathbb R^{n_1\times m}$ and $A_2\in\mathbb R^{n_2\times m}$, we define by $[A_1;A_2]\in\mathbb R^{(n_1+n_2)\times m}$ the matrix whose first $n_1$ rows are the rows of $A_1$ and the last $n_2$ rows are the rows of $A_2$.

We denote sets of numbers by calligraphic letters, e.g., $\mathcal I,\mathcal J$. The cardinality (size) of a set $\mathcal I$, $|\mathcal I|$ is the number of elements in the set. 
Furthermore, we denote by $\mathbb I_n^{\mathcal I}$, where $\mathcal I\subseteq\{1,\ldots,n\}$, the $n\times n$ matrix with the columns with indices in $\mathcal I$ equal to the columns of $\mathbb I_n$ and the remaining ones equal to zero. 
We use the semi-norm $\|\cdot\|_0$ function which counts the number of free parameters entries of a matrix, i.e., if $A\in\mathbb R^{n\times m}$ then $\|A\|_0=|\{A_{ij}\,:\,A_{ij}\neq 0,\text{ for }i=1,\ldots,n\text{ and }j=1,\ldots,m\}|$. 

A matrix $\bar M\in \{0,\star\}^{n\times m}$ is referred to as a \textit{structural matrix}. 
If $\bar M_{ij} = 0$, then $M_{ij}=0$, and if $\bar M_{ij}=\star$, then $M_{ij}\in\mathbb{R}$. 
Therefore,  if $\bar M_{ij}=\star$ then $M_{ij}$ is any arbitrary real number. 
Additionally, let $i\neq i'$ and $j\neq j'$, if $\bar M_{ij}=\star$ and $\bar M_{i'j'}=\star$ then $M_{ij}$ is assumed to be independent of $M_{i'j'}$. 
To simplify notation, given a structural matrix $\bar A\in\{0,\star\}^{n\times m}$ and $z\in\mathbb R$, we denote by $z \bar A\in\mathbb R^{n\times m}$ the matrix 
with the $\star$'s in $\bar A$ replaced by the number z.



Subsequently, we will make use of the following graph-theoretical notions. 
A \emph{digraph} (directed graph) is given by $\mathcal G=(\mathcal X,\mathcal E_{\mathcal X,\mathcal X})$, where $\mathcal X$ is a set of \emph{nodes} and $\mathcal E_{\mathcal X,\mathcal X}\subseteq \mathcal X\times\mathcal X$ is a set of \emph{edges} such that if $x_i,x_j\in\mathcal X$ and $(x_i,x_j)\in\mathcal E_{\mathcal X,\mathcal X}$ then there is an edge that starts in node $x_i$ and ends in node $x_j$. 
Given a structural matrix $\bar A\in\{0,\star\}^{n\times n}$, we associate to it the digraph representation $\mathcal G(\bar A)=(\mathcal X,\mathcal E_{\mathcal X,\mathcal X})$ such that $\mathcal X=\{x_1,\ldots,x_n\}$ and $\mathcal E_{\mathcal X,\mathcal X}=\{(x_i,x_j)\,:\,\bar A_{ji}\neq 0\}$.  

Given a digraph $\mathcal G=(\mathcal X,\mathcal E_{\mathcal X,\mathcal X})$, we define a \emph{path} from $x_1$ to $x_k$ with size $k$ as a sequence of nodes $(x_1,\ldots,x_k)$ such that $x_1,\ldots,x_k\in\mathcal X$, $x_i\neq x_j$ for $i\neq j$, and $(x_i,x_{i+1})\in\mathcal E_{\mathcal X,\mathcal X}$ for $i=1,\ldots,k-1$. 
A vertex with an edge to itself (i.e., a self-loop), or a path from $x_1$ to $x_k$ comprising an additional edge $(x_k,x_1)$, is called a \textit{cycle}.
A digraph is \emph{strongly connected} whenever there exists a path between each pair of nodes in the digraph. 

Additionally, we define a \emph{bipartite graph} as $\mathcal B=(\mathcal X_L,\mathcal X_R,\mathcal E_{\mathcal X_L,\mathcal X_R})$, where $\mathcal X_L\cup\mathcal X_R$ is the set of nodes with $\mathcal X_L\cap\mathcal X_R=\emptyset$, and $\mathcal E_{\mathcal X_L,\mathcal X_R}\subseteq \mathcal X_L\times\mathcal X_R$ is a set of edges. In other words, it is a graph with two disjoint sets of nodes such that there are only edges starting from nodes in the first set and ending in nodes of the second set. 
Moreover, we associate a structural matrix $\bar A\in\{0,\star\}^{n\times m}$ with a bipartite representation denoted by $\mathcal B(\bar A)=(\mathcal X_L,\mathcal X_R,\mathcal E_{\mathcal X_L,\mathcal X_R})$, where $X_L=\{x_1^L,\ldots,x_n^L\}$, $X_R=\{x_1^R,\ldots,x_m^R\}$, and $(x_i^L,x_j^R)\in \mathcal E_{\mathcal X_L,\mathcal X_R}$ whenever $\bar A_{ji}\neq 0$. 

In other words, we associated a bipartite graph where the second set of nodes is a virtual copy of the first. Additionally, the edges are represented as the original edges in $\mathcal G(\bar A)$, but where the starting node of an edge is in the first set of nodes and the ending node of an edge is in the second (virtual copy) of the nodes. 

Given a bipartite graph $\mathcal B=(\mathcal X_L,\mathcal X_R,\mathcal E_{\mathcal X_L,\mathcal X_R})$, a \emph{matching} $M\subseteq \mathcal E_{\mathcal X_L,\mathcal X_R}$ is a set of edges that do not share vertices, i.e., $(x,y),(x',y')\in M$ only if $x\neq x'$ and $y\neq y'$. A \emph{maximum matching} $M^\ast$ is a matching with the maximum possible number of edges. 
Given $\mathcal B=(\mathcal X_L,\mathcal X_R,\mathcal E_{\mathcal X_L,\mathcal X_R})$, the maximum matching problem can be solved with computational time complexity $\mathcal O(\sqrt{|\mathcal X_L\cup\mathcal X_R|}|\mathcal E_{\mathcal X_L,\mathcal X_R}|)$, which in the worst-case is $\mathcal O(\max\{|\mathcal X_L|,|\mathcal X_R|\}^{2.5})$~\cite{cormen2009introduction}. 
Furthermore, if we associate a weight $w_{ij}\in\mathbb{R}^+$ to each edge $e_{ij}$ of a bipartite graph, we may want to find a \emph{maximum weight maximum matching} (MWMM). 
In other words, a maximum matching with a maximal weight sum of the edges in the maximum matching. 
This problem can be solved utilizing, for instance, the Hungarian algorithm, with computational complexity $\mathcal O(\max\{|\mathcal X_L|,|\mathcal X_R|\}^3)$~\cite{cormen2009introduction}.

\section{Problem statement}\label{sec:prob_stat}
Consider a given (possibly large-scale) MADS described by the following linear time-invariant system (LTI) with autonomous dynamics  
\begin{equation}\label{eq:lti}
    x(k+1) = A x(k), 
\end{equation}
where $k\in\mathbb Z_0^+$, $x(k)\in\mathbb R^n$ denotes the state of the MADS, $A\in\mathbb R^{n\times n}$, and $x(0)=x_0$ is the initial state.

Given a system in~\eqref{eq:lti}, it is important to design matrices $B\in\mathbb R^{n\times p}$ and $C\in\mathbb R^{q\times n}$ so that \begin{equation}\label{eq:lti_2}
\begin{array}{rcl}
    x(k+1) &=& A x(k) + Bu(k),\\
    y(k)       &=& Cx(k),
\end{array}    
\end{equation}
is both controllable and observable, where $u(k)\in \mathbb R^p$ is the input signal, and $y(k)\in\mathbb R^q$ is the response of the system. To simplify the notation, we refer to~\eqref{eq:lti_2} as the triple $(A,B,C)$.  
Notice that~\eqref{eq:lti} and~\eqref{eq:lti_2} can be both posed in  continuous-time, as the controllability and observability criteria are the same.

Usually, for MADS, we have the freedom of selecting the dynamics weights of matrix $A$. Therefore, suppose that we only have available the sparsity pattern of $A$, i.e., the location of zeros and (possibly) non-zeros (free parameters) of the entries of $A$. 
When only the sparsity pattern is available, instead of designing matrices $B$ and $C$ that ensure controllability and observability of the system, we may design the structure of those matrices $\bar B$ and $\bar C$. 
In this case, the goal is to ensure structural controllability and structural observability of the triple $(\bar A,\bar B,\bar C)$~\cite{pequito2015framework}. 
Furthermore, it is common to use dedicated inputs and output in the context of MADS, as the actuators and sensors correspond to agents in the system that we actuate or observe.

Hence, the problem that we aim to solve in this paper is the following.

\noindent$\mathcal P_1$ Given a structural matrix $\bar A$ associated with~\eqref{eq:lti_2}, such that $\mathcal G(\bar A)$ is strongly connected, find 
\begin{equation}\label{eq:prob}
\begin{array}{rl}
    \mathcal I^\ast,\mathcal J^\ast & = \displaystyle\mathop{\arg\min}_{\mathcal I,\mathcal J\subseteq\{1,\ldots,n\}}|\mathcal I\cup\mathcal J|\\
    \text{s.t.} & (\bar A,\bar B=\bar{\mathbb I}_n^{\mathcal I},\bar C=\bar{\mathbb I}_n^{\mathcal J})\text{ is structurally}\\ & \text{controllable and observable,}
\end{array}
\end{equation}
where, for a set $\mathcal K\subset\{1,\ldots,n\}$, $\bar{\mathbb I}_n^{\mathcal K}\in\{0,\star\}^{n\times n}$ is a diagonal matrix such that $\bar{\mathbb I}^{\mathcal K}_{i,i}=\star$ whenever $i\in\mathcal K$.  

Observe that the requirement that a MADS is strongly connected is a common assumption in a plethora of applications~\cite{chi2020spatial,ramosIJoC2020,8028724}. 




A simple attempt to address problem $\mathcal P_1$ via decoupling it into the structural controllability and structural observability components would lead to the computation of all possible decoupled solutions to pinpoint a pair $(\mathcal I,\mathcal J)$ with a maximum intersection. Therefore, it would translate into a strictly combinatorial problem with a prohibitive computational complexity.

\section{Minimum jointly structural input and output selection for strongly connected networks}\label{sec:main}

Subsequently, we present necessary and sufficient conditions for the structural controllability and the structural observability of a system given $(\bar A,\bar B,\bar C)$.  
Recall that, in this paper, we are considering LTI systems that have a strongly connected system digraph representation. 
Therefore, we get the two lemmas of Theorem~3 from~\cite{pequito2015framework} stated below.

\begin{lemma}[\cite{pequito2015framework}]\label{lemma:struct_cont}
An LTI system~\eqref{eq:lti} with a strongly connected digraph representation is structurally controllable \underline{if and only if} 
    $\mathcal B\left( [\bar A, \bar B]\right)$ 
    has a maximum matching of size~$n$ and $\|B\|_0\geq 0$.
\hfill$\circ$
\end{lemma}

\begin{lemma}[\cite{pequito2015framework}]\label{lemma:struct_obs}
An LTI system~\eqref{eq:lti} with a strongly connected digraph representation is structurally observable \underline{if and only if} 
    $\mathcal B\left( [\bar A; \bar C]\right)$ 
    has a maximum matching of size~$n$ and $\|C\|_0\geq 0$.
\hfill$\circ$
\end{lemma}
Notice that the conditions $\|B\|_0\geq 0$ and $\|C\|_0\geq 0$ are only imposing that at least one input and one output should be placed to have a structurally controllable structural and structurally observable system. 

To overcome the identified computational intractability issue, we propose the following efficient (with polynomial-time complexity) algorithm -- see Algorithm~\ref{alg:main}. 

\begin{algorithm}[H]
		{
		\caption{Dedicated solution to $\mathcal P_1$}
		\label{alg:main}
		\begin{algorithmic}[1]
			\STATE{\textbf{input}: A structural dynamics matrix $\bar A$}
			\STATE{\textbf{output}: An input and output matrices, $\bar{\mathbb I}_n^{\mathcal I^\ast}$ and $\bar{\mathbb I}_n^{\mathcal J^\ast}$ respectively, describing a dedicated solution to $\mathcal P_1$}
			\STATE{\textbf{build} the bipartite graph $\mathcal B(\bar D)=(\mathcal X_L,\mathcal X_R,\mathcal E_{\mathcal X_L,\mathcal X_R})$, where  
			$$
			D =\begin{bmatrix}
            3\bar A   & \mathbb I_n\\
            \mathbb I_n & 2 \bar{\mathbb I}_n
            \end{bmatrix}
			.$$} 
			\STATE{\textbf{compute} $\mathcal M$ a MWMM of $\mathcal B(\bar D)$ with edges' weights given by $D$}
			\STATE{\textbf{set} $\mathcal I^\ast \,= \{(i,j)\in\mathcal M\,:\,i> n\land j\leq n\}$}
			\STATE{\textbf{set} $\mathcal J^\ast = \{(i,j)\in\mathcal M\,:\,i\leq n\land j > n\}$}
			\IF{$\mathcal I^\ast=\mathcal J^\ast=\emptyset$}
                \STATE{\textbf{select} one $i\in\{1,\ldots,n\}$}
                \STATE{\textbf{set} $\mathcal I^\ast=\mathcal J^\ast=\{i\}$}
            \ENDIF
		\end{algorithmic} 
		}
\end{algorithm} 

Intuitively, the first $n$ nodes of $\mathcal B(\bar D)$ correspond to the system state variables and the last $n$ nodes to inputs and output that we may activate. 
We can group the edges $(i,j)$ of $\mathcal B(\bar D)$ as follows:
\begin{itemize}
    \item $i,j\leq n$ that correspond to edges of $\mathcal G(\bar A)$;
    \item  $i>n$ and $j\leq n$ that represent connections between inputs and state variables;
    \item  $i\leq n$ and $j> n$ that represent connections between state variables and outputs;
    \item  $i> n$ and $j> n$ that represent connections between inputs and outputs (when chosen, the respective input and output are not selected to be used). 
\end{itemize} 
Then, Algorithm~\ref{alg:main} finds a maximum matching which matches the maximum possible number of nodes that correspond to state variables (corresponding to the $3\bar A$ part of $D$), while trying to place an input and an output to vertices that correspond to the same state variable (corresponding to the $\mathbb I_n$ parts of $D$). 
Moreover, this is done considering the use of the smallest possible number of inputs and outputs (corresponding to the $2\bar{\mathbb I}_n$ part of $D$). 
Only if it is not possible to assign an input and an output to nodes that correspond to the same state variables, different state variables are chosen.

\begin{theorem}\label{th:soundness}
Algorithm~\ref{alg:main} is sound, i.e., it computes a solution to problem $\mathcal P_1$.~\hfill$\circ$
\end{theorem}

\begin{proof}
First, we observe that $\mathcal I^\ast$ comprises a minimum set of dedicated inputs, which represents a maximum matching of size $n$ for the bipartite graph $\mathcal B([\bar A,\bar{\mathbb I}_n^{\mathcal I^\ast}])$, where $\bar{\mathbb I}_n^{\mathcal I^\ast}$ is a diagonal matrix whose entries in $\mathcal{I}^\ast$ are free parameters. Moreover, if the MWMM results in a perfect matching when restricted to $\mathcal B(\bar A)$, then $\mathcal I^\ast=\mathcal J^\ast=\emptyset$ and, in steps~7-9, we select any state variable to place both an input and an output, yielding a minimum value of $|\mathcal I^\ast\cup \mathcal J^\ast|=1$. As we are assuming that $\mathcal G(\bar A)$ is strongly connected, by Lemma~\ref{lemma:struct_cont} and Lemma~\ref{lemma:struct_obs}, the system is structurally controllable and structurally observable, respectively. 
Otherwise, we obtain the MMWM $\mathcal M$ of step~4, where we filter the edges to account only for connections between indices of state variables and indices of input variables in steps~5 and~6. 
In other words, $\mathcal M'=\{(i,j)\in\mathcal M\,:\,j\leq n\}$ is a maximum matching of $\mathcal B(\bar A)$. 
Hence, by Lemma~\ref{lemma:struct_cont}, $(\bar A,\bar B=\bar{\mathbb I}_n^{\mathcal I^\ast})$ is structurally controllable. 

Following a similar reasoning, $\mathcal J^\ast$ comprises a minimum set of dedicated outputs, which represents a MWMM of size $n$ for the bipartite graph $\mathcal B([\bar A;\bar{\mathbb I}_n^{\mathcal J^\ast}])$. This MWMM $\mathcal M''$ results from the MWMM $\mathcal M$ of step~4, and it is $\mathcal M''=\{(i,j)\in\mathcal M\,:\,i\leq n\}$. Therefore, by Lemma~\ref{lemma:struct_obs}, $(\bar A,\bar C=\bar{\mathbb I}_n^{\mathcal J^\ast})$ is structurally observable. 

Now, note that we know that the triple $(\bar A,\bar B=\bar{\mathbb I}_n^{\mathcal I^\ast},\bar C=\bar{\mathbb I}_n^{\mathcal J^\ast})$ is structurally controllable and structurally observable. Further, we need to check that the cost function $|\mathcal I^\ast\cup\mathcal J^\ast|$ is minimized with the solution found. 

In the creation of $\mathcal B(\bar D)$, we assigned weight 3 to the edges of $\mathcal B(\bar A)$, forcing those to be, preferably, selected to the MWMM. 
Moreover, we placed an edge with weight 2 between each dedicated input and dedicated output pair with the same index. 
By doing so, whenever it is possible to place both an input and an output to the same state variables, the MWMM of $\mathcal B(\bar D)$ increases because it matches another pair of input-output vertices which could not be pared before. 
Finally, the remaining edges have weight 1, forcing them to be only selected for the MWMM if there is no other option. 
In other words, a state variable only has an input and not an output (or vice-versa) if it cannot have both. 
Hence, the MWMM selects the maximum possible number of pairs input-output to actuate and observe the same state variable. 
\end{proof}

Next, we analyse the computational complexity of Algorithm~\ref{alg:main}. 

\begin{proposition}\label{prop:complexity}
The worst-case computational time-complexity of Algorithm~\ref{alg:main} is $\mathcal O(n^{3})$.\hfill$\circ$ 
\end{proposition}

\begin{proof}
Step~5 can be solved using the Hungarian algorithm~\cite{cormen2009introduction}, which finds a MWMM of $\mathcal B(\bar D)$ with time-complexity $\mathcal O(\max\{|\mathcal X_L|,|\mathcal X_R|\}|^3 )$. 
Since $|\mathcal X_L|=|\mathcal X_R|=2n=\mathcal O(n)$, then the time-complexity of step~5 is $\mathcal O(n^{3})$.  
\end{proof}

Note that we can obtain an approximated solution in almost linear-time (in the number of vertices and edges of the associated system's digraph) if we allow obtaining approximated MWMM in Algorithm~\ref{alg:main}.  
For example, we may use~\cite{duan2014linear} which allows us to obtain a $(1-\varepsilon)$-approximation of the solution (for any specified $\varepsilon>0$), with time complexity, that depend on $\varepsilon$, of $\mathcal O\left(M\frac{1}{\varepsilon}\log\frac{1}{\varepsilon}\right)$ (i.e., linear time), where $M$ is the number of edges of $\mathcal B(\bar D)=(\mathcal X_L,\mathcal X_R,\mathcal E_{\mathcal X_L,\mathcal X_R})$ built in Algorithm~\ref{alg:main}.


In the next section, we illustrate the proposed method with examples, and compare it with the simple approach that only aims to find minimal dedicated input and output placements (without necessarily maximizing the intersections between the two).

\section{Illustrative examples}\label{sec:ill_exp}

In this section, we explore three examples. The first two correspond to structural matrices representing MADSs with bidirectional communication networks. The last one represents a unidirectional communication network.  
We compare the proposed approach against solving the structural controllability and observability parts separately. In the three examples, we achieve a solution that uses a smaller number of actuated/observed state variables than the separate solution. 

\subsection{Example 1}
To illustrate how Algorithm~\ref{alg:main} works, consider a structural matrix 
\[\bar A=
\begin{bmatrix}
0 & \star & 0\\
\star & 0 & \star\\
0 & \star & 0
\end{bmatrix},
\]
whose digraph representation is depicted in Figure~\ref{fig:exemplo}. 

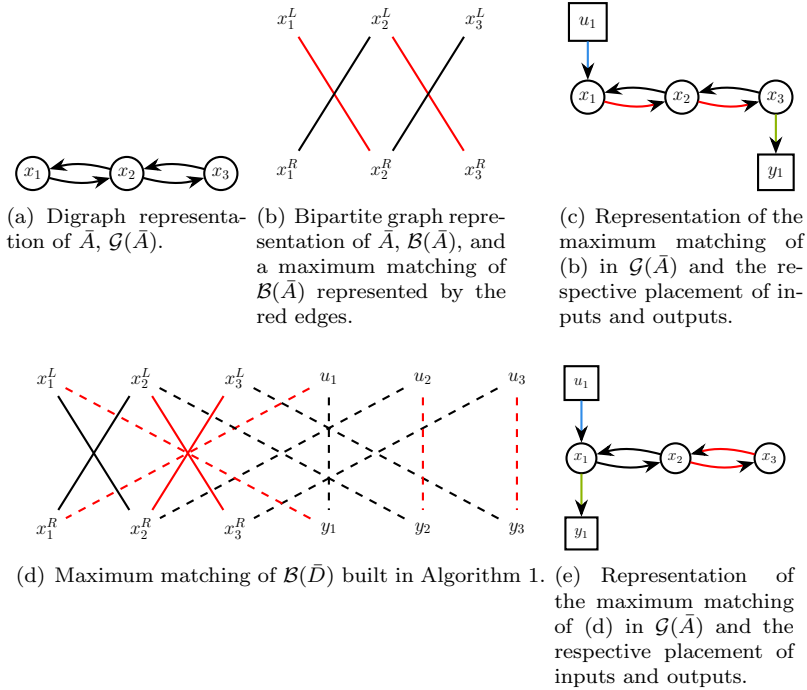
\begin{figure}[H]
\centering
\subfigure[Digraph representation of $\bar A$, $\mathcal G(\bar A)$.]{
\begin{tikzpicture}[scale=.5, transform shape,node distance=1.5cm]
\begin{scope}[every node/.style={circle,thick,draw},square/.style={regular polygon,regular polygon sides=4}]
\node (1) at (0.,0.) {\Large $x_1$};
\node (2) at (2.50076,0.) {\Large $x_2$};
\node (3) at (5.,0.) {\Large $x_3$};
\end{scope}
\begin{scope}[>={Stealth[black]},
              every edge/.style={draw=black, thick}]
\path [->] (1) edge[bend right=15] node {} (2);
\path [->] (2) edge[bend right=15] node {} (1);
\path [->] (2) edge[bend right=15] node {} (3);
\path [->] (3) edge[bend right=15] node {} (2);
\end{scope}
\end{tikzpicture}
}
\subfigure[Bipartite graph representation of $\bar A$, $\mathcal B(\bar A)$, and a maximum matching of $\mathcal B(\bar A)$ represented by the red edges.]{
\begin{tikzpicture}[scale=.5, transform shape,node distance=1.5cm]
\begin{scope}[every node/.style={circle,thick,draw=white},square/.style={regular polygon,regular polygon sides=4}]
\node (1) at (0.,2) {\Large $x_1^L$};
\node (2) at (2.5,2) {\Large $x_2^L$};
\node (3) at (5.,2) {\Large $x_3^L$};
\node (1 ) at (0.,-2) {\Large $x_1^R$};
\node (2 ) at (2.5,-2) {\Large $x_2^R$};
\node (3 ) at (5.,-2) {\Large $x_3^R$};
\end{scope}
\begin{scope}[>={Stealth[black]},
              every edge/.style={draw=black, thick}]
\path [-] (1) edge[red] node {} (2 );
\path [-] (2) edge node {} (1 );
\path [-] (2) edge[red] node {} (3 );
\path [-] (3) edge node {} (2 );
\end{scope}
\end{tikzpicture}
}\qquad
\subfigure[Representation of the maximum matching of (b) in $\mathcal G(\bar A)$ and the respective placement of inputs and outputs.]{
\begin{tikzpicture}[scale=.5, transform shape,node distance=1.5cm]
\begin{scope}[every node/.style={circle,thick,draw},square/.style={regular polygon,regular polygon sides=4}]
\node (1) at (0.,0.) {\Large $x_1$};
\node (2) at (2.50076,0.) {\Large $x_2$};
\node (3) at (5.,0.) {\Large $x_3$};
\node[square,draw] (u) at (0.,2) {\Large $u_1$};
\node[square,draw] (y) at (5.,-2) {\Large $y_1$};
\end{scope}
\begin{scope}[>={Stealth[black]},
              every edge/.style={draw=black, thick}]
\path [->] (1) edge[red,bend right=15] node {} (2);
\path [->] (2) edge[bend right=15] node {} (1);
\path [->] (2) edge[red,bend right=15] node {} (3);
\path [->] (3) edge[bend right=15] node {} (2);
\path [->] (3) edge[ao] node {} (y);
\path [->] (u) edge[bleudefrance] node {} (1);

\end{scope}
\end{tikzpicture}
}

\subfigure[Maximum matching of $\mathcal B(\bar D)$ built in Algorithm~\ref{alg:main}.]{
\begin{tikzpicture}[scale=.5, transform shape,node distance=1.5cm]
\begin{scope}[every node/.style={circle,thick,draw=white},square/.style={regular polygon,regular polygon sides=4}]
\node (1) at (0.,2) {\Large $x_1^L$};
\node (2) at (2.5,2) {\Large $x_2^L$};
\node (3) at (5.,2) {\Large $x_3^L$};
\node (1 ) at (0.,-2) {\Large $x_1^R$};
\node (2 ) at (2.5,-2) {\Large $x_2^R$};
\node (3 ) at (5.,-2) {\Large $x_3^R$};
\node[square,draw] (u1) at (7.5,2) {\Large $u_1$};
\node[square,draw] (u2) at (10,2) {\Large $u_2$};
\node[square,draw] (u3) at (12.5,2) {\Large $u_3$};
\node[square,draw] (y1) at (7.5,-2) {\Large $y_1$};
\node[square,draw] (y2) at (10,-2) {\Large $y_2$};
\node[square,draw] (y3) at (12.5,-2) {\Large $y_3$};
\end{scope}
\begin{scope}[>={Stealth[black]},
              every edge/.style={draw=black, thick}]
\path [-] (1) edge node {} (2 );
\path [-] (2) edge node {} (1 );
\path [-] (2) edge[red] node {} (3 );
\path [-] (3) edge[red] node {} (2 );
\path [-] (u1) edge[red,dashed] node {} (1 );
\path [-] (u2) edge[dashed] node {} (2 );
\path [-] (u3) edge[dashed] node {} (3 );
\path [-] (1) edge[red,dashed] node {} (y1);
\path [-] (2) edge[dashed] node {} (y2);
\path [-] (3) edge[dashed] node {} (y3);
\path [-] (u1) edge[dashed] node {} (y1);
\path [-] (u2) edge[red,dashed] node {} (y2);
\path [-] (u3) edge[red,dashed] node {} (y3);
\end{scope}
\end{tikzpicture}
}
\subfigure[Representation of the maximum matching of (d) in $\mathcal G(\bar A)$ and the respective placement of inputs and outputs.]{
\begin{tikzpicture}[scale=.5, transform shape,node distance=1.5cm]
\begin{scope}[every node/.style={circle,thick,draw},square/.style={regular polygon,regular polygon sides=4}]
\node (1) at (0.,0.) {\large $x_1$};
\node (2) at (2.50076,0.) {\large $x_2$};
\node (3) at (5.,0.) {\large $x_3$};
\node[square,draw] (u) at (0.,2) {\large $u_1$};
\node[square,draw] (y) at (0.,-2) {\large $y_1$};
\end{scope}
\begin{scope}[>={Stealth[black]},
              every edge/.style={draw=black, thick}]
\path [->] (1) edge[bend right=15] node {} (2);
\path [->] (2) edge[bend right=15] node {} (1);
\path [->] (2) edge[red,bend right=15] node {} (3);
\path [->] (3) edge[red,bend right=15] node {} (2);
\path [->] (1) edge[ao] node {} (y);
\path [->] (u) edge[bleudefrance] node {} (1);
\end{scope}
\end{tikzpicture}
}
\caption{Illustration of Algorithm~\ref{alg:main}.}
\label{fig:exemplo}
\end{figure}

Observe that the solution obtained with Algorithm~\ref{alg:main} is minimal, $|\mathcal I^\ast\cup\mathcal J^\ast|=1$ -- Figure~\ref{fig:exemplo}~(e) -- and the solution achieve with the previous methods is not minimal, $|\mathcal I^\ast\cup\mathcal J^\ast|=2$ -- Figure~\ref{fig:exemplo}~(c).


\subsection{Example 2}
Consider the structural matrix 
\[
\bar A_1 = 
\begin{bmatrix}
0 & 0 & \star & 0 & \star & 0 & 0 & 0 & \star & \star \\
0 & 0 & \star & 0 & 0 & 0 & 0 & 0 & 0 & 0 \\
\star & \star & 0 & 0 & 0 & \star & \star & 0 & 0 & 0\\
0 & 0 & 0 & 0 & \star & 0 & 0 & 0 & 0 & 0 \\
\star & 0 & 0 & \star & 0 & 0 & 0 & \star & 0 & 0 \\
0 & 0 & \star & 0 & 0 & 0 & 0 & 0 & 0 & 0 \\
0 & 0 & \star & 0 & 0 & 0 & 0 & 0 & 0 & 0 \\
0 & 0 & 0 & 0 & \star & 0 & 0 & 0 & 0 & 0 \\
\star & 0 & 0 & 0 & 0 & 0 & 0 & 0 & 0 & 0 \\
\star & 0 & 0 & 0 & 0 & 0 & 0 & 0 & 0 & 0 \\
\end{bmatrix},
\]
with digraph representation $\mathcal G(\bar A_1)$ depicted in Figure~\ref{fig:MWMM}~(a). 
A maximum matching of the bipartite graph representation $\mathcal B(\bar A_1)$ is depicted in Figure~\ref{fig:MWMM}~(b), and it corresponds to red edges in Figure~\ref{fig:MWMM}~(c) that yields an input placed at each state variable in $\mathcal I=\{x_2,x_6,x_8,x_{10}\}$, and an output placed at each state variable in $\mathcal J=\{x_4,x_6,x_7,x_9\}$. 
The cost function of $\mathcal P_1$ \underline{is not minimal}, $|\mathcal I\cup\mathcal J|=6$, as we detail next. 
Algorithm~\ref{alg:main} yields dedicated input and output placements to $\mathcal I=\mathcal J=\{x_6,x_7,x_8,x_{10}\}$ -- see Figure~\ref{fig:MWMM_2}~(a) and~(b). 
Hence, the cost function of $\mathcal P_1$ \underline{is minimal}, i.e., $|\mathcal I\cup\mathcal J|=4$.

\begin{figure}[H]
\centering
\subfigure[Digraph representation $\mathcal G(\bar A_1)$.]{
\begin{tikzpicture}[scale=.45, transform shape,node distance=1.5cm]
\begin{scope}[every node/.style={circle,thick,draw},square/.style={regular polygon,regular polygon sides=4}]
\node (1) at (9.86388,4.7212) {\Large $x_1$};
\node (3) at (5.21248,6.08031) {\Large $x_3$};
\node (5) at (14.166,6.30483) {\Large $x_5$};
\node (9) at (8.87808,1.9768) {\Large $x_9$};
\node (10) at (11.4592,2.12937) {\Large $x_{10}$};
\node (2) at (2.78625,4.08182) {\Large $x_2$};
\node (6) at (2.05665,6.84987) {\Large $x_6$};
\node (7) at (3.9797,8.98236) {\Large $x_7$};
\node (4) at (17.3146,5.68281) {\Large $x_4$};
\node (8) at (16.2332,8.70226) {\Large $x_8$};
\end{scope}
\begin{scope}[>={Stealth[black]},
              every edge/.style={draw=black, thick}]
\path [->] (1) edge[bend right=10] node {} (3);
\path [->] (3) edge[bend right=10] node {} (1);
\path [->] (1) edge[bend right=10] node {} (5);
\path [->] (5) edge[bend right=10] node {} (1);
\path [->] (1) edge[bend right=10] node {} (9);
\path [->] (9) edge[bend right=10] node {} (1);
\path [->] (1) edge[bend right=10] node {} (10);
\path [->] (10) edge[bend right=10] node {} (1);
\path [->] (2) edge[bend right=10] node {} (3);
\path [->] (3) edge[bend right=10] node {} (2);
\path [->] (3) edge[bend right=10] node {} (6);
\path [->] (6) edge[bend right=10] node {} (3);
\path [->] (3) edge[bend right=10] node {} (7);
\path [->] (7) edge[bend right=10] node {} (3);
\path [->] (4) edge[bend right=10] node {} (5);
\path [->] (5) edge[bend right=10] node {} (4);
\path [->] (5) edge[bend right=10] node {} (8);
\path [->] (8) edge[bend right=10] node {} (5);
\end{scope}
\end{tikzpicture}
}
\subfigure[Bipartite graph representation $\mathcal B(\bar A_1)$, with a maximum matching depicted by the red edges.]{
\begin{tikzpicture}[scale=.6, transform shape,node distance=1.5cm]
\begin{scope}[every node/.style={circle,thick,draw=white},square/.style={regular polygon,regular polygon sides=4}]
\node (1) at (0,2) {\Large $x_1^L$};
\node (2) at (1,2) {\Large $x_2^L$};
\node (3) at (2,2) {\Large $x_3^L$};
\node (4) at (3,2) {\Large $x_4^L$};
\node (5) at (4,2) {\Large $x_5^L$};
\node (6) at (5,2) {\Large $x_6^L$};
\node (7) at (6,2) {\Large $x_7^L$};
\node (8) at (7,2) {\Large $x_8^L$};
\node (9) at (8,2) {\Large $x_9^L$};
\node (10) at (9,2) {\Large $x_{10}^L$};
\node (1 ) at (0,-2) {\Large $x_1^R$};
\node (2 ) at (1,-2) {\Large $x_2^R$};
\node (3 ) at (2,-2) {\Large $x_3^R$};
\node (4 ) at (3,-2) {\Large $x_4^R$};
\node (5 ) at (4,-2) {\Large $x_5^R$};
\node (6 ) at (5,-2) {\Large $x_6^R$};
\node (7 ) at (6,-2) {\Large $x_7^R$};
\node (8 ) at (7,-2) {\Large $x_8^R$};
\node (9 ) at (8,-2) {\Large $x_9^R$};
\node (10 ) at (9,-2) {\Large $x_{10}^R$};
\end{scope}
\begin{scope}[>={Stealth[black]},
              every edge/.style={draw=black}]
\path [-] (1) edge node {} (3 );
\path [-] (3) edge node {} (1 );
\path [-] (1) edge node {} (5 );
\path [-] (5) edge node {} (1 );
\path [-] (1) edge[red] node {} (9 );
\path [-] (9) edge node {} (1 );
\path [-] (1) edge node {} (10 );
\path [-] (10) edge[red] node {} (1 );
\path [-] (2) edge[red] node {} (3 );
\path [-] (3) edge node {} (2 );
\path [-] (3) edge node {} (6 );
\path [-] (6) edge node {} (3 );
\path [-] (3) edge[red] node {} (7 );
\path [-] (7) edge node {} (3 );
\path [-] (4) edge node {} (5 );
\path [-] (5) edge[red] node {} (4 );
\path [-] (5) edge node {} (8 );
\path [-] (8) edge[red] node {} (5 );

\end{scope}
\end{tikzpicture}
}
\subfigure[Dedicated input and output placements, obtained with~\cite{pequito2015framework} (without accounting for the intersection of state variables that are actuated and sensed): $\mathcal I=\{x_2,x_6,x_8,x_{10}\}$, $\mathcal J=\{x_4,x_6,x_7,x_9\}$. The cost function of $\mathcal P_1$ \underline{is not minimal}, $|\mathcal I\cup\mathcal J|=6$.]{
\begin{tikzpicture}[scale=.45, transform shape,node distance=1.5cm]
\begin{scope}[every node/.style={circle,thick,draw},square/.style={regular polygon,regular polygon sides=4}]
\node (1) at (9.86388,4.7212) {\Large $x_1$};
\node (3) at (5.21248,6.08031) {\Large $x_3$};
\node (5) at (14.166,6.30483) {\Large $x_5$};
\node (9) at (8.87808,1.9768) {\Large $x_9$};
\node (10) at (11.4592,2.12937) {\Large $x_{10}$};
\node (2) at (2.78625,4.08182) {\Large $x_2$};
\node (6) at (2.05665,6.84987) {\Large $x_6$};
\node (7) at (3.9797,8.98236) {\Large $x_7$};
\node (4) at (17.3146,5.68281) {\Large $x_4$};
\node (8) at (16.2332,8.70226) {\Large $x_8$};

\node[square,draw] (u1) at (2.78625,1.9768) {\Large $u_1$};
\node[square,draw] (u2) at (1,5) {\Large $u_2$};
\node[square,draw] (u3) at (14.166,8.70226) {\Large $u_3$};
\node[square,draw] (u4) at (13.8,2.12937) {\Large $u_4$};

\node[square,draw] (y1) at (17.3146,3.7) {\Large $y_1$};
\node[square,draw] (y2) at (1,8.65) {\Large $y_2$};
\node[square,draw] (y3) at (6.2,8.98236) {\Large $y_3$};
\node[square,draw] (y4) at (6.7,1.9768) {\Large $y_4$};


\end{scope}
\begin{scope}[>={Stealth[black]},
              every edge/.style={draw=black, thick}]
\path [->] (1) edge[bend right=10] node {} (3);
\path [->] (3) edge[bend right=10] node {} (1);
\path [->] (1) edge[bend right=10] node {} (5);
\path [->] (5) edge[bend right=10] node {} (1);
\path [->] (1) edge[red,bend right=10] node {} (9);
\path [->] (9) edge[bend right=10] node {} (1);
\path [->] (1) edge[bend right=10] node {} (10);
\path [->] (10) edge[red,bend right=10] node {} (1);
\path [->] (2) edge[red,bend right=10] node {} (3);
\path [->] (3) edge[bend right=10] node {} (2);
\path [->] (3) edge[bend right=10] node {} (6);
\path [->] (6) edge[bend right=10] node {} (3);
\path [->] (3) edge[red,bend right=10] node {} (7);
\path [->] (7) edge[bend right=10] node {} (3);
\path [->] (4) edge[bend right=10] node {} (5);
\path [->] (5) edge[red,bend right=10] node {} (4);
\path [->] (5) edge[bend right=10] node {} (8);
\path [->] (8) edge[red,bend right=10] node {} (5);

\path [->] (u1) edge[bleudefrance] node {} (2);
\path [->] (u2) edge[bleudefrance] node {} (6);
\path [->] (u3) edge[bleudefrance] node {} (8);
\path [->] (u4) edge[bleudefrance] node {} (10);

\path [->] (4) edge[ao] node {} (y1);
\path [->] (6) edge[ao] node {} (y2);
\path [->] (7) edge[ao] node {} (y3);
\path [->] (9) edge[ao] node {} (y4);

\end{scope}
\end{tikzpicture}
}
\caption{Illustrative example 1.}
\label{fig:MWMM}
\end{figure}
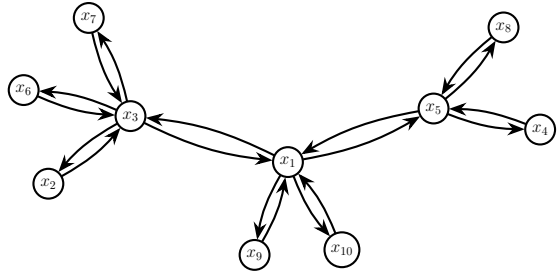
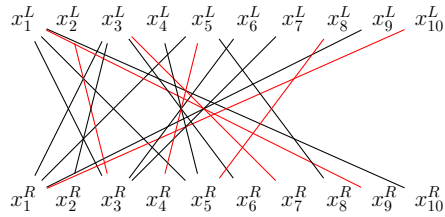
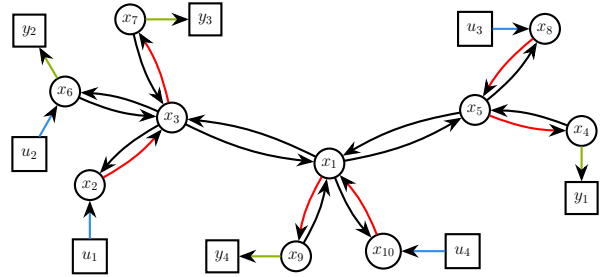

\begin{figure}[H]
\centering
\subfigure[MWMM of $\mathcal B(D)$, with edges' weights given by $D$, built with Algorithm~\ref{alg:main} for the input $\bar A_1$.]{
\begin{tikzpicture}[scale=.55, transform shape,node distance=1.5cm]
\begin{scope}[every node/.style={circle,thick,draw=white},square/.style={regular polygon,regular polygon sides=4}]
\node (1) at (0,2) {\Large $x_1^L$};
\node (2) at (1,2) {\Large $x_2^L$};
\node (3) at (2,2) {\Large $x_3^L$};
\node (4) at (3,2) {\Large $x_4^L$};
\node (5) at (4,2) {\Large $x_5^L$};
\node (6) at (5,2) {\Large $x_6^L$};
\node (7) at (6,2) {\Large $x_7^L$};
\node (8) at (7,2) {\Large $x_8^L$};
\node (9) at (8,2) {\Large $x_9^L$};
\node (10) at (9,2) {\Large $x_{10}^L$};
\node (1 ) at (0,-2.5) {\Large $x_1^R$};
\node (2 ) at (1,-2.5) {\Large $x_2^R$};
\node (3 ) at (2,-2.5) {\Large $x_3^R$};
\node (4 ) at (3,-2.5) {\Large $x_4^R$};
\node (5 ) at (4,-2.5) {\Large $x_5^R$};
\node (6 ) at (5,-2.5) {\Large $x_6^R$};
\node (7 ) at (6,-2.5) {\Large $x_7^R$};
\node (8 ) at (7,-2.5) {\Large $x_8^R$};
\node (9 ) at (8,-2.5) {\Large $x_9^R$};
\node (10 ) at (9,-2.5) {\Large $x_{10}^R$};
\node[square,draw] (u1) at (10,2) {\Large $u_1$};
\node[square,draw] (u2) at (11,2) {\Large $u_2$};
\node[square,draw] (u3) at (12,2) {\Large $u_3$};
\node[square,draw] (u4) at (13,2) {\Large $u_4$};
\node[square,draw] (u5) at (14,2) {\Large $u_5$};
\node[square,draw] (u6) at (15,2) {\Large $u_6$};
\node[square,draw] (u7) at (16,2) {\Large $u_7$};
\node[square,draw] (u8) at (17,2) {\Large $u_8$};
\node[square,draw] (u9) at (18,2) {\Large $u_9$};
\node[square,draw] (u10) at (19,2) {\Large $u_{10}$};
\node[square,draw] (y1) at (10,-2.5) {\Large $y_1$};
\node[square,draw] (y2) at (11,-2.5) {\Large $y_2$};
\node[square,draw] (y3) at (12,-2.5) {\Large $y_3$};
\node[square,draw] (y4) at (13,-2.5) {\Large $y_4$};
\node[square,draw] (y5) at (14,-2.5) {\Large $y_5$};
\node[square,draw] (y6) at (15,-2.5) {\Large $y_6$};
\node[square,draw] (y7) at (16,-2.5) {\Large $y_7$};
\node[square,draw] (y8) at (17,-2.5) {\Large $y_8$};
\node[square,draw] (y9) at (18,-2.5) {\Large $y_9$};
\node[square,draw] (y10) at (19,-2.5) {\Large $y_{10}$};
\end{scope}
\begin{scope}[>={Stealth[black]},
              every edge/.style={draw=black}]
\path [-] (1) edge node {} (3 );
\path [-] (3) edge node {} (1 );
\path [-] (1) edge node {} (5 );
\path [-] (5) edge node {} (1 );
\path [-] (1) edge[red] node {} (9 );
\path [-] (9) edge[red] node {} (1 );
\path [-] (1) edge node {} (10 );
\path [-] (10) edge node {} (1 );
\path [-] (2) edge[red] node {} (3 );
\path [-] (3) edge[red] node {} (2 );
\path [-] (3) edge node {} (6 );
\path [-] (6) edge node {} (3 );
\path [-] (3) edge node {} (7 );
\path [-] (7) edge node {} (3 );
\path [-] (4) edge[red] node {} (5 );
\path [-] (5) edge[red] node {} (4 );
\path [-] (5) edge node {} (8 );
\path [-] (8) edge node {} (5 );

\path [-] (u1) edge[dashed] node {} (1 );
\path [-] (u2) edge[dashed] node {} (2 );
\path [-] (u3) edge[dashed] node {} (3 );
\path [-] (u4) edge[dashed] node {} (4 );
\path [-] (u5) edge[dashed] node {} (5 );
\path [-] (u6) edge[red,dashed] node {} (6 );
\path [-] (u7) edge[red,dashed] node {} (7 );
\path [-] (u8) edge[red,dashed] node {} (8 );
\path [-] (u9) edge[dashed] node {} (9 );
\path [-] (u10) edge[red,dashed] node {} (10 );
\path [-] (1) edge[dashed] node {} (y1);
\path [-] (2) edge[dashed] node {} (y2);
\path [-] (3) edge[dashed] node {} (y3);
\path [-] (4) edge[dashed] node {} (y4);
\path [-] (5) edge[dashed] node {} (y5);
\path [-] (6) edge[red,dashed] node {} (y6);
\path [-] (7) edge[red,dashed] node {} (y7);
\path [-] (8) edge[red,dashed] node {} (y8);
\path [-] (9) edge[dashed] node {} (y9);
\path [-] (10) edge[red,dashed] node {} (y10);
\path [-] (u1) edge[red,dashed] node {} (y1);
\path [-] (u2) edge[red,dashed] node {} (y2);
\path [-] (u3) edge[red,dashed] node {} (y3);
\path [-] (u4) edge[red,dashed] node {} (y4);
\path [-] (u5) edge[red,dashed] node {} (y5);
\path [-] (u6) edge[dashed] node {} (y6);
\path [-] (u7) edge[dashed] node {} (y7);
\path [-] (u8) edge[dashed] node {} (y8);
\path [-] (u9) edge[red,dashed] node {} (y9);
\path [-] (u10) edge[dashed] node {} (y10);
\end{scope}
\end{tikzpicture}
}
\subfigure[Dedicated input and output placements, obtained with Algorithm~\ref{alg:main}$: \mathcal I=\mathcal J=\{x_6,x_7,x_8,x_{10}\}$. The cost function of $\mathcal P_1$ \underline{is minimal}, $|\mathcal I\cup\mathcal J|=4$.]{
\begin{tikzpicture}[scale=.45, transform shape,node distance=1.5cm]
\begin{scope}[every node/.style={circle,thick,draw},square/.style={regular polygon,regular polygon sides=4}]
\node (1) at (9.86388,4.7212) {\Large $x_1$};
\node (3) at (5.21248,6.08031) {\Large $x_3$};
\node (5) at (14.166,6.30483) {\Large $x_5$};
\node (9) at (8.87808,1.9768) {\Large $x_9$};
\node (10) at (11.4592,2.12937) {\Large $x_{10}$};
\node (2) at (2.78625,4.08182) {\Large $x_2$};
\node (6) at (2.05665,6.84987) {\Large $x_6$};
\node (7) at (3.9797,8.98236) {\Large $x_7$};
\node (4) at (17.3146,5.68281) {\Large $x_4$};
\node (8) at (16.2332,8.70226) {\Large $x_8$};

\node[square,draw] (u1) at (0,6.84987) {\Large $u_1$};
\node[square,draw] (u2) at (1.8,8.98236) {\Large $u_2$};
\node[square,draw] (u3) at (14.1,8.70226) {\Large $u_3$};
\node[square,draw] (u4) at (13.6,2.12937) {\Large $u_4$};

\node[square,draw] (y1) at (0.9,5.) {\Large $y_1$};
\node[square,draw] (y2) at (6.1,8.98236) {\Large $y_2$};
\node[square,draw] (y3) at (18.2332,8.70226) {\Large $y_3$};
\node[square,draw] (y4) at (13,4.12937) {\Large $y_4$};

\end{scope}
\begin{scope}[>={Stealth[black]},
              every edge/.style={draw=black, thick}]
\path [->] (1) edge[bend right=10] node {} (3);
\path [->] (3) edge[bend right=10] node {} (1);
\path [->] (1) edge[bend right=10] node {} (5);
\path [->] (5) edge[bend right=10] node {} (1);
\path [->] (1) edge[red,bend right=10] node {} (9);
\path [->] (9) edge[red,bend right=10] node {} (1);
\path [->] (1) edge[bend right=10] node {} (10);
\path [->] (10) edge[bend right=10] node {} (1);
\path [->] (2) edge[red,bend right=10] node {} (3);
\path [->] (3) edge[red,bend right=10] node {} (2);
\path [->] (3) edge[bend right=10] node {} (6);
\path [->] (6) edge[bend right=10] node {} (3);
\path [->] (3) edge[bend right=10] node {} (7);
\path [->] (7) edge[bend right=10] node {} (3);
\path [->] (4) edge[red,bend right=10] node {} (5);
\path [->] (5) edge[red,bend right=10] node {} (4);
\path [->] (5) edge[bend right=10] node {} (8);
\path [->] (8) edge[bend right=10] node {} (5);

\path [->] (u1) edge[bleudefrance] node {} (6);
\path [->] (u2) edge[bleudefrance] node {} (7);
\path [->] (u3) edge[bleudefrance] node {} (8);
\path [->] (u4) edge[bleudefrance] node {} (10);

\path [->] (6) edge[ao] node {} (y1);
\path [->] (7) edge[ao] node {} (y2);
\path [->] (8) edge[ao] node {} (y3);
\path [->] (10) edge[ao] node {} (y4);

\end{scope}
\end{tikzpicture}
}
\caption{Illustrative example 1: input and output placement using Algorithm~\ref{alg:main}.}
\label{fig:MWMM_2}
\end{figure}
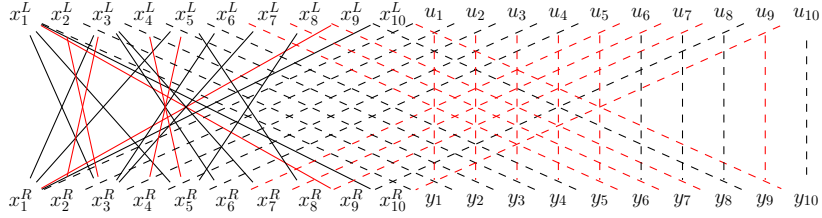
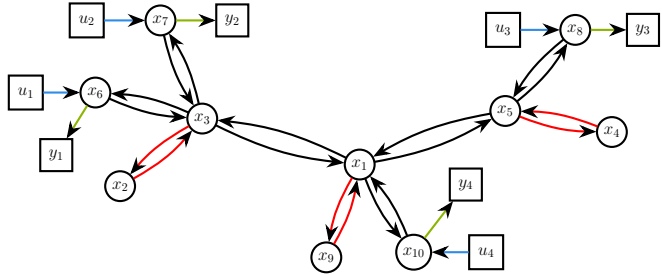

\subsection{Example 3}

In the next example, we consider a MADS with a strongly connected network, where the edges do not correspond to bidirectional communication as in the two previous examples. 
We consider the structural matrix 
\[
\bar A_2=\begin{bmatrix}
0 & 0 & 0 & 0 & 0 & 0 & \star & 0 & 0 & 0\\
0 & 0 & \star & 0 & 0 & \star & 0 & 0 & 0 & 0\\
0 & 0 & 0 & 0 & 0 & \star & 0 & 0 & 0 & \star\\
0 & 0 & \star & 0 & 0 & 0 & 0 & 0 & \star & 0\\
0 & 0 & 0 & 0 & 0 & 0 & 0 & 0 & \star & 0\\
0 & 0 & 0 & 0 & 0 & 0 & 0 & 0 & 0 & \star\\
\star & 0 & 0 & 0 & 0 & 0 & 0 & \star & 0 & 0\\
0 & 0 & 0 & \star & 0 & 0 & 0 & 0 & 0 & 0\\
0 & 0 & 0 & \star & \star & 0 & 0 & 0 & 0 & 0\\
\star & \star & \star & 0 & 0 & 0 & 0 & 0 & 0 & 0\\
\end{bmatrix},
\]
with digraph representation $\mathcal G(\bar A_2)$ depicted in Figure~\ref{fig:MWMM2}~(a). 
A maximum matching of the bipartite graph representation $\mathcal B(\bar A_2)$ is depicted in Figure~\ref{fig:MWMM2}~(b), and it corresponds to the red edges in Figure~\ref{fig:MWMM2}~(c) that yields an input placed at each state variable in $\mathcal I=\{x_6\}$, and an output placed at each state variable in $\mathcal J=\{x_8\}$. 
The cost function of $\mathcal P_1$ \underline{is not minimal}, $|\mathcal I\cup\mathcal J|=2$, as we explore next. 
By using Algorithm~\ref{alg:main}, we obtain a placement of dedicated inputs and outputs to the same state variables, $\mathcal I=\mathcal J=\{x_2\}$, see Figure~\ref{fig:MWMM2b}~(a) and~(b). 
Now, the cost function of $\mathcal P_1$ \underline{is minimal}, $|\mathcal I\cup\mathcal J|=1$. 

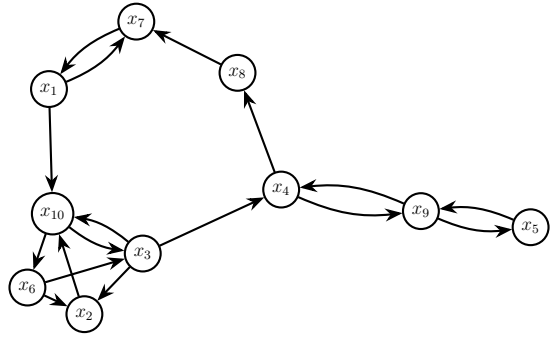
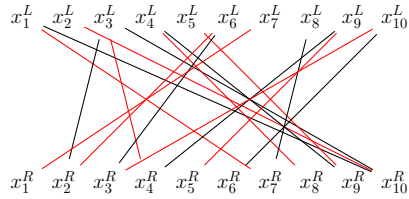
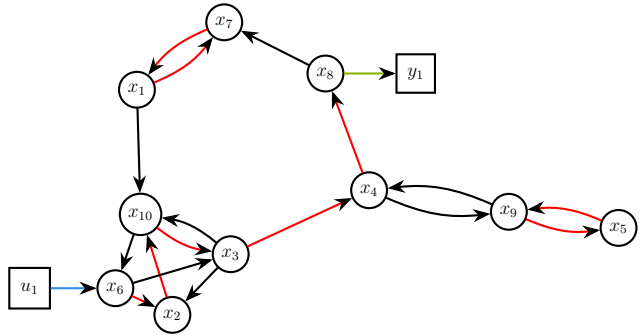
\begin{figure}[H]
\centering
\subfigure[Digraph representation $\mathcal G(\bar A_2)$.]{
\begin{tikzpicture}[scale=.6, transform shape,node distance=1.5cm]
\begin{scope}[every node/.style={circle,thick,draw},square/.style={regular polygon,regular polygon sides=4}]
\node (1) at (0.377759,4.99195) {\large $x_1$};
\node (2) at (1.16242,0.) {\large $x_2$};
\node (3) at (2.45522,1.34463) {\large $x_3$};
\node (4) at (5.52224,2.76325) {\large $x_4$};
\node (5) at (11.0489,1.92791) {\large $x_5$};
\node (6) at (-0.1,0.590678) {\large $x_6$};
\node (7) at (2.31223,6.48719) {\large $x_7$};
\node (8) at (4.55499,5.35156) {\large $x_8$};
\node (9) at (8.61862,2.28536) {\large $x_9$};
\node (10) at (0.45846,2.22768) {\large $x_{10}$};
\end{scope}
\begin{scope}[>={Stealth[black]},
              every edge/.style={draw=black, thick}]
\path [->] (1) edge[bend right=15] node {} (7);
\path [->] (1) edge node {} (10);
\path [->] (2) edge node {} (10);
\path [->] (3) edge node {} (2);
\path [->] (3) edge node {} (4);
\path [->] (3) edge[bend right=15] node {} (10);
\path [->] (4) edge node {} (8);
\path [->] (4) edge[bend right=15] node {} (9);
\path [->] (5) edge[bend right=15] node {} (9);
\path [->] (6) edge node {} (2);
\path [->] (6) edge node {} (3);
\path [->] (7) edge[bend right=15] node {} (1);
\path [->] (8) edge node {} (7);
\path [->] (9) edge[bend right=15] node {} (4);
\path [->] (9) edge[bend right=15] node {} (5);
\path [->] (10) edge[bend right=15] node {} (3);
\path [->] (10) edge node {} (6);
\end{scope}
\end{tikzpicture}
}

\subfigure[Bipartite graph representation $\mathcal B(\bar A_2)$, with a maximum matching depicted by the red edges.]{
\begin{tikzpicture}[scale=.55, transform shape,node distance=1.5cm]
\begin{scope}[every node/.style={circle,thick,draw=white},square/.style={regular polygon,regular polygon sides=4}]
\node (1) at (0,2) {\Large $x_1^L$};
\node (2) at (1,2) {\Large $x_2^L$};
\node (3) at (2,2) {\Large $x_3^L$};
\node (4) at (3,2) {\Large $x_4^L$};
\node (5) at (4,2) {\Large $x_5^L$};
\node (6) at (5,2) {\Large $x_6^L$};
\node (7) at (6,2) {\Large $x_7^L$};
\node (8) at (7,2) {\Large $x_8^L$};
\node (9) at (8,2) {\Large $x_9^L$};
\node (10) at (9,2) {\Large $x_{10}^L$};
\node (1 ) at (0,-2) {\Large $x_1^R$};
\node (2 ) at (1,-2) {\Large $x_2^R$};
\node (3 ) at (2,-2) {\Large $x_3^R$};
\node (4 ) at (3,-2) {\Large $x_4^R$};
\node (5 ) at (4,-2) {\Large $x_5^R$};
\node (6 ) at (5,-2) {\Large $x_6^R$};
\node (7 ) at (6,-2) {\Large $x_7^R$};
\node (8 ) at (7,-2) {\Large $x_8^R$};
\node (9 ) at (8,-2) {\Large $x_9^R$};
\node (10 ) at (9,-2) {\Large $x_{10}^R$};
\end{scope}
\begin{scope}[>={Stealth[black]},
              every edge/.style={draw=black}]
\path [-] (1) edge[red] node {} (7 );
\path [-] (1) edge node {} (10 );
\path [-] (2) edge[red] node {} (10 );
\path [-] (3) edge node {} (2 );
\path [-] (3) edge[red] node {} (4 );
\path [-] (3) edge node {} (10 );
\path [-] (4) edge[red] node {} (8 );
\path [-] (4) edge node {} (9 );
\path [-] (5) edge[red] node {} (9 );
\path [-] (6) edge[red] node {} (2 );
\path [-] (6) edge node {} (3 );
\path [-] (7) edge[red] node {} (1 );
\path [-] (8) edge node {} (7 );
\path [-] (9) edge node {} (4 );
\path [-] (9) edge[red] node {} (5 );
\path [-] (10) edge[red] node {} (3 );
\path [-] (10) edge node {} (6 );

\end{scope}
\end{tikzpicture}
}

\subfigure[Dedicated input and output placements, obtained with~\cite{pequito2015framework} (without accounting for the intersection of state variables that are actuated and sensed): $\mathcal I=\{x_6\}$, $\mathcal J=\{x_8\}$. The cost function of $\mathcal P_1$ \underline{is not minimal}, $|\mathcal I\cup\mathcal J|=2$.]{
\begin{tikzpicture}[scale=.6, transform shape,node distance=1.5cm]
\begin{scope}[every node/.style={circle,thick,draw},square/.style={regular polygon,regular polygon sides=4}]
\node (1) at (0.377759,4.99195) {\large $x_1$};
\node (2) at (1.16242,0.) {\large $x_2$};
\node (3) at (2.45522,1.34463) {\large $x_3$};
\node (4) at (5.52224,2.76325) {\large $x_4$};
\node (5) at (11.0489,1.92791) {\large $x_5$};
\node (6) at (-0.1,0.590678) {\large $x_6$};
\node (7) at (2.31223,6.48719) {\large $x_7$};
\node (8) at (4.55499,5.35156) {\large $x_8$};
\node (9) at (8.61862,2.28536) {\large $x_9$};
\node (10) at (0.45846,2.22768) {\large $x_{10}$};

\node[square,draw] (u1) at (-2,0.590678) {\large $u_1$};

\node[square,draw] (y1) at (6.55499,5.35156) {\large $y_1$};

\end{scope}
\begin{scope}[>={Stealth[black]},
              every edge/.style={draw=black, thick}]
\path [->] (1) edge[red,bend right=15] node {} (7);
\path [->] (1) edge node {} (10);
\path [->] (2) edge[red] node {} (10);
\path [->] (3) edge node {} (2);
\path [->] (3) edge[red] node {} (4);
\path [->] (3) edge[bend right=15] node {} (10);
\path [->] (4) edge[red] node {} (8);
\path [->] (4) edge[bend right=15] node {} (9);
\path [->] (5) edge[red,bend right=15] node {} (9);
\path [->] (6) edge[red] node {} (2);
\path [->] (6) edge node {} (3);
\path [->] (7) edge[red,bend right=15] node {} (1);
\path [->] (8) edge node {} (7);
\path [->] (9) edge[bend right=15] node {} (4);
\path [->] (9) edge[red,bend right=15] node {} (5);
\path [->] (10) edge[red,bend right=15] node {} (3);
\path [->] (10) edge node {} (6);

\path [->] (u1) edge[bleudefrance] node {} (6);

\path [->] (8) edge[ao] node {} (y1);

\end{scope}
\end{tikzpicture}
}
\caption{Illustrative example 2.}
\label{fig:MWMM2}
\end{figure}

\begin{figure}[H]
\centering
\subfigure[MWMM of $\mathcal B(D)$, with edges' weights given by $D$, built with Algorithm~\ref{alg:main} for the input $\bar A_2$.]{
\begin{tikzpicture}[scale=.55, transform shape,node distance=1.5cm]
\begin{scope}[every node/.style={circle,thick,draw=white},square/.style={regular polygon,regular polygon sides=4}]
\node (1) at (0,2) {\Large $x_1^L$};
\node (2) at (1,2) {\Large $x_2^L$};
\node (3) at (2,2) {\Large $x_3^L$};
\node (4) at (3,2) {\Large $x_4^L$};
\node (5) at (4,2) {\Large $x_5^L$};
\node (6) at (5,2) {\Large $x_6^L$};
\node (7) at (6,2) {\Large $x_7^L$};
\node (8) at (7,2) {\Large $x_8^L$};
\node (9) at (8,2) {\Large $x_9^L$};
\node (10) at (9,2) {\Large $x_{10}^L$};
\node (1 ) at (0,-2.5) {\Large $x_1^R$};
\node (2 ) at (1,-2.5) {\Large $x_2^R$};
\node (3 ) at (2,-2.5) {\Large $x_3^R$};
\node (4 ) at (3,-2.5) {\Large $x_4^R$};
\node (5 ) at (4,-2.5) {\Large $x_5^R$};
\node (6 ) at (5,-2.5) {\Large $x_6^R$};
\node (7 ) at (6,-2.5) {\Large $x_7^R$};
\node (8 ) at (7,-2.5) {\Large $x_8^R$};
\node (9 ) at (8,-2.5) {\Large $x_9^R$};
\node (10 ) at (9,-2.5) {\Large $x_{10}^R$};
\node[square,draw] (u1) at (10,2) {\Large $u_1$};
\node[square,draw] (u2) at (11,2) {\Large $u_2$};
\node[square,draw] (u3) at (12,2) {\Large $u_3$};
\node[square,draw] (u4) at (13,2) {\Large $u_4$};
\node[square,draw] (u5) at (14,2) {\Large $u_5$};
\node[square,draw] (u6) at (15,2) {\Large $u_6$};
\node[square,draw] (u7) at (16,2) {\Large $u_7$};
\node[square,draw] (u8) at (17,2) {\Large $u_8$};
\node[square,draw] (u9) at (18,2) {\Large $u_9$};
\node[square,draw] (u10) at (19,2) {\Large $u_{10}$};
\node[square,draw] (y1) at (10,-2.5) {\Large $y_1$};
\node[square,draw] (y2) at (11,-2.5) {\Large $y_2$};
\node[square,draw] (y3) at (12,-2.5) {\Large $y_3$};
\node[square,draw] (y4) at (13,-2.5) {\Large $y_4$};
\node[square,draw] (y5) at (14,-2.5) {\Large $y_5$};
\node[square,draw] (y6) at (15,-2.5) {\Large $y_6$};
\node[square,draw] (y7) at (16,-2.5) {\Large $y_7$};
\node[square,draw] (y8) at (17,-2.5) {\Large $y_8$};
\node[square,draw] (y9) at (18,-2.5) {\Large $y_9$};
\node[square,draw] (y10) at (19,-2.5) {\Large $y_{10}$};
\end{scope}
\begin{scope}[>={Stealth[black]},
              every edge/.style={draw=black}]
\path [-] (1) edge node {} (7 );
\path [-] (1) edge[red] node {} (10 );
\path [-] (2) edge node {} (10 );
\path [-] (3) edge node {} (2 );
\path [-] (3) edge[red] node {} (4 );
\path [-] (3) edge node {} (10 );
\path [-] (4) edge[red] node {} (8 );
\path [-] (4) edge node {} (9 );
\path [-] (5) edge[red] node {} (9 );
\path [-] (6) edge node {} (2 );
\path [-] (6) edge[red] node {} (3 );
\path [-] (7) edge[red] node {} (1 );
\path [-] (8) edge[red] node {} (7 );
\path [-] (9) edge node {} (4 );
\path [-] (9) edge[red] node {} (5 );
\path [-] (10) edge node {} (3 );
\path [-] (10) edge[red] node {} (6 );

\path [-] (u1) edge[dashed] node {} (1 );
\path [-] (u2) edge[red,dashed] node {} (2 );
\path [-] (u3) edge[dashed] node {} (3 );
\path [-] (u4) edge[dashed] node {} (4 );
\path [-] (u5) edge[dashed] node {} (5 );
\path [-] (u6) edge[dashed] node {} (6 );
\path [-] (u7) edge[dashed] node {} (7 );
\path [-] (u8) edge[dashed] node {} (8 );
\path [-] (u9) edge[dashed] node {} (9 );
\path [-] (u10) edge[dashed] node {} (10 );
\path [-] (1) edge[dashed] node {} (y1);
\path [-] (2) edge[red,dashed] node {} (y2);
\path [-] (3) edge[dashed] node {} (y3);
\path [-] (4) edge[dashed] node {} (y4);
\path [-] (5) edge[dashed] node {} (y5);
\path [-] (6) edge[dashed] node {} (y6);
\path [-] (7) edge[dashed] node {} (y7);
\path [-] (8) edge[dashed] node {} (y8);
\path [-] (9) edge[dashed] node {} (y9);
\path [-] (10) edge[dashed] node {} (y10);
\path [-] (u1) edge[red,dashed] node {} (y1);
\path [-] (u2) edge[dashed] node {} (y2);
\path [-] (u3) edge[red,dashed] node {} (y3);
\path [-] (u4) edge[red,dashed] node {} (y4);
\path [-] (u5) edge[red,dashed] node {} (y5);
\path [-] (u6) edge[red,dashed] node {} (y6);
\path [-] (u7) edge[red,dashed] node {} (y7);
\path [-] (u8) edge[red,dashed] node {} (y8);
\path [-] (u9) edge[red,dashed] node {} (y9);
\path [-] (u10) edge[red,dashed] node {} (y10);
\end{scope}
\end{tikzpicture}
}
\subfigure[Dedicated input and output placements, obtained with Algorithm~\ref{alg:main}$: \mathcal I=\mathcal J=\{x_2\}$. The cost function of $\mathcal P_1$ \underline{is minimal}, $|\mathcal I\cup\mathcal J|=1$.]{
\begin{tikzpicture}[scale=.6, transform shape,node distance=1.5cm]
\begin{scope}[every node/.style={circle,thick,draw},square/.style={regular polygon,regular polygon sides=4}]
\node (1) at (0.377759,4.99195) {\large $x_1$};
\node (2) at (1.16242,0.) {\large $x_2$};
\node (3) at (2.45522,1.34463) {\large $x_3$};
\node (4) at (5.52224,2.76325) {\large $x_4$};
\node (5) at (11.0489,1.92791) {\large $x_5$};
\node (6) at (-0.1,0.590678) {\large $x_6$};
\node (7) at (2.31223,6.48719) {\large $x_7$};
\node (8) at (4.55499,5.35156) {\large $x_8$};
\node (9) at (8.61862,2.28536) {\large $x_9$};
\node (10) at (0.45846,2.22768) {\large $x_{10}$};

\node[square,draw] (u1) at (-0.8,-0.4) {\large $u_1$};

\node[square,draw] (y1) at (3.06242,-0.2) {\large $y_1$};

\end{scope}
\begin{scope}[>={Stealth[black]},
              every edge/.style={draw=black, thick}]
\path [->] (1) edge[bend right=15] node {} (7);
\path [->] (1) edge[red] node {} (10);
\path [->] (2) edge node {} (10);
\path [->] (3) edge node {} (2);
\path [->] (3) edge[red] node {} (4);
\path [->] (3) edge[bend right=15] node {} (10);
\path [->] (4) edge[red] node {} (8);
\path [->] (4) edge[bend right=15] node {} (9);
\path [->] (5) edge[red,bend right=15] node {} (9);
\path [->] (6) edge node {} (2);
\path [->] (6) edge[red] node {} (3);
\path [->] (7) edge[red,bend right=15] node {} (1);
\path [->] (8) edge[red] node {} (7);
\path [->] (9) edge[bend right=15] node {} (4);
\path [->] (9) edge[red,bend right=15] node {} (5);
\path [->] (10) edge[bend right=15] node {} (3);
\path [->] (10) edge[red] node {} (6);

\path [->] (u1) edge[bleudefrance] node {} (2);

\path [->] (2) edge[ao] node {} (y1);

\end{scope}
\end{tikzpicture}
}
\caption{ Input  and output placement using Algorithm~\ref{alg:main} for  Illustrative example 2.}
\label{fig:MWMM2b}
\end{figure}


\section{Conclusions}\label{sec:conclusion}

This paper studies the problem of given a MADS with strongly connected network, represented as an LTI system, identifying a minimal set of state variables to be actuated and a minimal set of state variables to be measured that achieve a maximum intersection while ensuring structural controllability and structural observability. 
We present a solution to the problem that couples the design of both structural controllability and structural observability counterparts, which has $\mathcal O(n^{3})$ (i.e., polynomial) time-complexity. 

Future work includes extending the proposed framework to the scenario where the network given by the dynamics matrix is not strongly connected, and to explore if it is possible to efficiently design solutions that account for robustness to input and output failures.

\subsection*{Code availability}
An implementation of Algorithm~\ref{alg:main}, using the Wolfram Mathematica\textsuperscript{\textregistered} programming language, is available via GitHub\footnote{\url{https://github.com/xuizy/structural_control/tree/joint_input_output}}.

\bibliography{bib.bib}

\begin{thebibliography}{10}
\providecommand{\url}[1]{#1}
\csname url@samestyle\endcsname
\providecommand{\newblock}{\relax}
\providecommand{\bibinfo}[2]{#2}
\providecommand{\BIBentrySTDinterwordspacing}{\spaceskip=0pt\relax}
\providecommand{\BIBentryALTinterwordstretchfactor}{4}
\providecommand{\BIBentryALTinterwordspacing}{\spaceskip=\fontdimen2\font plus
\BIBentryALTinterwordstretchfactor\fontdimen3\font minus
  \fontdimen4\font\relax}
\providecommand{\BIBforeignlanguage}[2]{{%
\expandafter\ifx\csname l@#1\endcsname\relax
\typeout{** WARNING: IEEEtran.bst: No hyphenation pattern has been}%
\typeout{** loaded for the language `#1'. Using the pattern for}%
\typeout{** the default language instead.}%
\else
\language=\csname l@#1\endcsname
\fi
#2}}
\providecommand{\BIBdecl}{\relax}
\BIBdecl

\bibitem{dorri2018multi}
A.~Dorri, S.~S. Kanhere, and R.~Jurdak, ``Multi-agent systems: A survey,''
  \emph{{IEEE Access}}, vol.~6, pp. 28\,573--28\,593, 2018.

\bibitem{9304107}
G.~{Ramos}, D.~{Silvestre}, and C.~{Silvestre}, ``A general discrete-time
  method to achieve resilience in consensus algorithms,'' in \emph{2020 59th
  IEEE Conference on Decision and Control (CDC)}, 2020, pp. 2702--2707.

\bibitem{ramosIJoC2020}
G.~Ramos, D.~Silvestre, and C.~Silvestre, ``{General Resilient Consensus
  Algorithms},'' \emph{{International Journal of Control}}, vol.~0, no.~ja, pp.
  1--27, 2020.

\bibitem{hu2019distributed}
J.~Hu, P.~Bhowmick, and A.~Lanzon, ``Distributed adaptive time-varying group
  formation tracking for multiagent systems with multiple leaders on directed
  graphs,'' \emph{IEEE Transactions on Control of Network Systems}, vol.~7,
  no.~1, pp. 140--150, 2019.

\bibitem{luo2018distributed}
F.~Luo, Z.~Y. Dong, G.~Liang, J.~Murata, and Z.~Xu, ``A distributed electricity
  trading system in active distribution networks based on multi-agent coalition
  and blockchain,'' \emph{{IEEE Transactions on Power Systems}}, vol.~34,
  no.~5, pp. 4097--4108, 2018.

\bibitem{RAMOS2021104842}
G.~Ramos, D.~Silvestre, and C.~Silvestre, ``Node and network resistance to
  bribery in multi-agent systems,'' \emph{Systems \& Control Letters}, vol.
  147, p. 104842, 2020.

\bibitem{nadi2017adaptive}
A.~Nadi and A.~Edrisi, ``Adaptive multi-agent relief assessment and emergency
  response,'' \emph{International journal of disaster risk reduction}, vol.~24,
  pp. 12--23, 2017.

\bibitem{akyildiz2002wireless}
I.~F. Akyildiz, W.~Su, Y.~Sankarasubramaniam, and E.~Cayirci, ``Wireless sensor
  networks: a survey,'' \emph{Computer networks}, vol.~38, no.~4, pp. 393--422,
  2002.

\bibitem{rribeiro:controlo}
R.~Ribeiro, D.~Silvestre, and C.~Silvestre, ``Decentralized control for
  multi-agent missions based on flocking rules,'' in \emph{CONTROLO 2020},
  J.~A. Gon{\c{c}}alves, M.~Braz-C{\'e}sar, and J.~P. Coelho, Eds.\hskip 1em
  plus 0.5em minus 0.4em\relax Cham: Springer International Publishing, 2021,
  pp. 445--454.

\bibitem{rribeiro:med}
R.~{Ribeiro}, D.~{Silvestre}, and C.~{Silvestre}, ``A rendezvous algorithm for
  multi-agent systems in disconnected network topologies,'' in \emph{2020 28th
  Mediterranean Conference on Control and Automation (MED)}.\hskip 1em plus
  0.5em minus 0.4em\relax 28th Mediterranean Conference on Control and
  Automation (MED), 2020, pp. 592--597.

\bibitem{kaminer2007coordinated}
I.~Kaminer, O.~Yakimenko, V.~Dobrokhodov, A.~Pascoal, N.~Hovakimyan, V.~Patel,
  C.~Cao, and A.~Young, ``Coordinated path following for time-critical missions
  of multiple uavs via l1 adaptive output feedback controllers,'' in \emph{AIAA
  Guidance, Navigation and Control Conference and Exhibit}, 2007, p. 6409.

\bibitem{olshevsky2014minimal}
A.~Olshevsky, ``Minimal controllability problems,'' \emph{{IEEE Transactions on
  Control of Network Systems}}, vol.~1, no.~3, pp. 249--258, 2014.

\bibitem{pequito2017robust}
S.~Pequito, G.~Ramos, S.~Kar, A.~P. Aguiar, and J.~Ramos, ``The robust minimal
  controllability problem,'' \emph{Automatica}, vol.~82, pp. 261--268, 2017.

\bibitem{ramos2018robust}
G.~Ramos, S.~Pequito, and C.~Caleiro, ``The robust minimal controllability
  problem for switched linear continuous-time systems,'' in \emph{{2018 Annual
  American Control Conference (ACC)}}.\hskip 1em plus 0.5em minus 0.4em\relax
  {IEEE}, 2018, pp. 210--215.

\bibitem{ramos2021robust}
G.~Ramos, D.~Silvestre, and C.~Silvestre, ``The robust minimal controllability
  and observability problem,'' \emph{International Journal of Robust and
  Nonlinear Control}, 2021.

\bibitem{ramos2020structural}
G.~Ramos, A.~P. Aguiar, and S.~Pequito, ``Structural systems theory: an
  overview of the last 15 years,'' \emph{arXiv preprint arXiv:2008.11223},
  2020.

\bibitem{lin1974structural}
C.-T. Lin, ``Structural controllability,'' \emph{{IEEE} {T}ransactions on
  {A}utomatic {C}ontrol}, vol.~19, no.~3, pp. 201--208, 1974.

\bibitem{pequito2015framework}
S.~Pequito, S.~Kar, and A.~P. Aguiar, ``A framework for structural input/output
  and control configuration selection in large-scale systems,'' \emph{IEEE
  Transactions on Automatic Control}, vol.~61, no.~2, pp. 303--318, 2015.

\bibitem{cormen2009introduction}
T.~H. Cormen, C.~E. Leiserson, R.~L. Rivest, and C.~Stein, \emph{Introduction
  to algorithms}.\hskip 1em plus 0.5em minus 0.4em\relax MIT press, 2009.

\bibitem{chi2020spatial}
R.~Chi, Y.~Hui, B.~Huang, Z.~Hou, and X.~Bu, ``Spatial linear dynamic
  relationship of strongly connected multiagent systems and adaptive learning
  control for different formations,'' \emph{IEEE transactions on cybernetics},
  2020.

\bibitem{8028724}
D.~Wang, Z.~Wang, D.~Wang, and W.~Wang, ``Distributed optimization for
  multiagent systems over general strongly connected digraph,'' in \emph{2017
  36th Chinese Control Conference (CCC)}, 2017, pp. 8613--8620.

\bibitem{duan2014linear}
R.~Duan and S.~Pettie, ``Linear-time approximation for maximum weight
  matching,'' \emph{Journal of the {ACM} ({JACM})}, vol.~61, no.~1, pp. 1--23,
  2014.

\end{thebibliography}

\end{document}